\documentclass[12pt]{amsart}
\usepackage{amssymb, amscd, amsmath, amsthm, epsf, epsfig, latexsym, enumerate}

\newtheorem{theorem}{Theorem}
\newtheorem{lemma}[theorem]{Lemma}
\newtheorem*{cor}{Corollary}

\begin{document}

\title{The groups of fibred 2-knots}

\author{Jonathan A. Hillman }
\address{School of Mathematics and Statistics F07\\
     University of Sydney, Sydney\\ 
     NSW 2006, Australia }

\email{jonh@maths.usyd.edu.au}

\begin{abstract}
We explore algebraic characterizations of 2-knots whose 
associated knot manifolds fibre over lower-dimensional orbifolds, 
and consider also some issues related to the groups of 
higher-dimensional fibred knots.
\end{abstract}



\maketitle
Nontrivial classical knot groups have cohomological dimension 2,
and the knot is fibred if and only if the commutator subgroup 
is finitely generated, 
in which case the commutator subgroup is free of even rank.
Poincar\'e duality and the condition $\chi(M(K))=0$ 
together impose subtle constraints on 2-knot groups which 
do not apply in higher dimensions.
In particular, if the commutator subgroup $\pi'$ 
of a 2-knot group $\pi$ is finitely generated then 
the virtual cohomological dimension of $\pi$ is 1, 2 or 4.
In this note we shall show that (modulo several plausible conjectures)
a 2-knot group $\pi$ is the group of a fibred 2-knot if and only if $\pi'$ 
is finitely generated, and if moreover $\pi$ is torsion-free
every 2-knot with group $\pi$ is $s$-concordant to a fibred 2-knot.
A simple satellite construction gives an example of a 2-knot 
whose group $\pi$ is not virtually torsion-free
(and so $\pi'$ is not finitely generated).

Although our main interest is in the case of 2-knots,
we give examples of fibred $n$-knots with groups of 
cohomological dimension $d$, for every $n\geq4$ and $d\geq1$.
(No purely algebraic characterization of the groups of fibred $n$-knots
is yet known for any $n$.)
It is not clear whether there are fibred 3-knots with such groups.
In the final section we consider other possible fibrations of 2-knot manifolds.

This work was prompted by reading \cite{GGS}, where it is 
shown that there is a high-dimensional knot group which contains 
copies of every finitely presentable group,
and it is suggested that there should be a similar 2-knot group.
Our results do not address the questions raised at
the end of \cite{GGS} beyond the observations that 
no examples supporting the suggestions made there
can be the groups of fibred 2-knots,
and very likely no such examples 
have finitely generated commutator subgroup.

\newpage
\section{fibred $1$-knots}

Let $\pi$ be the group of a fibred 1-knot $K$,
and let $t\in\pi$ represent a meridian of the knot.
Then $\pi'$ is free of finite rank $2g$, and the meridianal automorphism
$\phi$ determined by conjugation by $t$ is geometric:
it is induced by an orientation-preserving self-homeomorphism $f$
of $T_{g,o}$, the once-punctured surface with $g$ handles.

Conversely, let $f$ be an orientation-preserving self-homeomorphism 
of $T_{g,o}$ which fixes $\partial{T_{g,o}}$.
The mapping torus $M(f)$ has fundamental group
$\pi=F(2g)\rtimes_{f_*}Z$ and boundary a torus.
Let $h:S^1\times{S^1}\to\partial{M(f)}$ be a homeomorphism such that
$\mu=h(\{*\}\times{S^1})$ is a section of the projection to $S^1$.
If $\pi$ has weight 1 we may assume that $\mu$
represents a normal generator for $\pi$.
Then $\Sigma=M(f)\cup_h{S^1\times{D^2}}$ is a homotopy 3-sphere,
and $K=h|_{S^1\times\{0\}}$ is a fibred knot with exterior 
homeomorphic to $M(f)$.

Is there an algebraic characterization of 
such geometric outer automorphism classes?
The situation is simpler in higher dimensions.

\begin{lemma}
Let $\pi$ be a finitely presentable group with $\pi'$ 
finitely generated and $\pi/\pi'\cong{Z}$.
Then the following are equivalent
\begin{enumerate}
\item $\pi'\cong{F(r)}$ for some $r\geq0$;

\item $c.d.\pi\leq2$ and $\pi'$ is $FP_2$;

\item $\pi$ has deficiency $1$.
\end{enumerate}
If these conditions hold and $\pi$ has weight $1$
then it is the group of a fibred $n$-knot, for all $n\geq2$.
\end{lemma}

\begin{proof}
It is easy to see that $(1)\Rightarrow(2)$ and (3).
Conversely, (2) and (3) each imply that $\pi'$ is free,
by Corollary 8.6 of \cite{Bi} and by the ``Rapaport Conjecture" \cite{Ko}, 
respectively.

If these conditions hold then $\pi\cong{F(r)\rtimes_\theta{Z}}$,
where $\theta$ is the automorphism induced by conjugation by 
a normal generator for $\pi$.
We may realize $\theta$ by a basepoint
preserving self-homeomorphism
$h$ of $\#^r(S^n\times{S^1})$, for every $n\geq2$.
If $\pi$ has weight 1 then the cocore of surgery on
the section of the mapping torus $M(h)$ determined by
the basepoint is a fibred $n$-knot with group $\pi$.
\end{proof}

Is it sufficient to assume that $c.d.\pi=2$ and $\pi'$ is finitely generated?

\section{fibred $2$-knots}

If $K$ is a 2-knot with $\pi'$ finitely presentable then
$M(K)'$ is a $PD_3$-complex \cite{HK},
and the indecomposable factors of
$\pi'$ are $PD_3$-groups or are virtually free \cite{Cr}.
Hence $v.c.d.\pi=1$, 2 or 4.
The main result of this section is contingent upon 
the following two ASSUMPTIONS:  
\begin{enumerate}
\item all $PD_3$-groups are 3-manifold groups; 
\item if the fundamental group of an indecomposable $PD_3$-complex
is virtually free and maps onto $T_1^*$ then it is finite.
\end{enumerate}

A finitely generated group $\nu$ has a decomposition
${(*_{i\in{I}}G_i)*F(s)}$, 
where the groups $G_i$ are indecomposable but not $Z$.
The non-free factors are unique up to permutation and conjugacy in $\nu$,
by the Grushko-Neumann Theorem and the Kurosh Subgroup Theorem.
Automorphisms of $\nu$ are generated by automorphisms of the factors,
permutations of isomorphic factors, 
conjugacy and ``generalized Whitehead moves",
corresponding to dragging a summand around a loop.
The latter two types of automorphism induce the identity 
on the abelianizations of the non-free factors $G_i$.
(This analysis of $Aut(\nu)$ is due to D.I.Fuchs-Rabinovitch. 
See \cite{Gi} for a more recent account.)

\begin{lemma}
Let $G=*_{i\in{Z/rZ}}G_i$, where the factors are isomorphic,
and let $f$ be an automorphism of $G$ such that $f(G_i)=G_{i+1}$, 
for all $i\in{Z/rZ}$.
Then $f$ is meridianal if and only if the restriction of
$f^r$ to $G_1$ is meridianal.
Similarly, $H_1(f)-1$ is an isomorphism of $H_1(G;\mathbb{Z})$
if and only if 
$H_1(f^r)-1$ is an isomorphism of $H_1(G_1;\mathbb{Z})$.
\end{lemma} 

\begin{proof}
This is clear, since
\[G/\langle\langle
{g^{-1}f(g)\mid{g\in{G}}}
\rangle\rangle\cong
G_1/\langle\langle{g^{-1}f^r(g)\mid{g\in{G_1}}}\rangle\rangle,
\]
and similarly for the abelianization.
\end{proof}

The semidirect product $G\rtimes_f{Z}$ has a presentation
\[\langle{G_1,t}\mid {t^rgt^{-r}}=f^r(g)~\forall{g\in{G_1}}\rangle.\]
If $G_1=\pi_1(N)$ where $N$ is an $\mathbb{S}^3$-manifold
or an aspherical 3-manifold and $f^r$ is meridianal and 
is realizable by a self-homeomorphism of $N$
then $G\rtimes_f{Z}$ is the group of an indecomposable 
fibred 2-knot with fibre $\sharp_{i=1}^rN$.
In particular, if $(s,d)=1$ and $r\geq1$ then there is 
a fibred 2-knot with fibre $\sharp_{i=1}^rL(d,s)$ 
and group having presentation
\[\langle{a,t}\mid{a^d=1,}~t^rat^{-r}=a^s\rangle.\]
If $K$ is a fibred 2-knot such that $\pi'$ has no nontrivial free factor and
the meridianal automorphism is in the subgroup generated by 
factor automorphisms, permutations and conjugacy then
$K$ is a connected sum of indecomposable 2-knots with groups
of the type just described.
Thus such knots provide basic building blocks for fibred 2-knots.

\begin{theorem}
Let $\pi$ be a $2$-knot group. If the above assumptions hold 
the following are equivalent:
\begin{enumerate}
\item $\pi=\pi{K}$ where $K$ is fibred;
\item there is a closed orientable $3$-manifold $N$ and a
meridianal automorphism $\theta$ of $\nu=\pi_1(N)$ such that
$\theta_*c_{N*}([N])=c_{N*}([N])$ and $\pi\cong\nu\rtimes_\theta{Z}$;
\item $\pi'$ is finitely generated.
\end{enumerate}
\end{theorem} 

\begin{proof}
The implications $(1)\Rightarrow(2)$  and $(2)\Rightarrow(3)$ are clear.

Suppose that $\pi'$ is finitely generated.
Then $M'$ is a $PD_3$-space and $\pi'$ is $FP_2$ \cite{HK}.
Let $\pi'=(*_{i\in{I}}G_i)*F(s)$ be a factorization of $\pi'$
in which the factors $G_i$ are indecomposable but not free.
The arguments of Crisp and Turaev apply equally well here to show that
if $G_i$ has one end it is a $PD_3$-group, 
and otherwise it is virtually free.
By the above assumptions, the $PD_3$-group factors 
are the fundamental groups of aspherical closed 3-manifolds and 
the remaining non-free factors have abelianization $Z/2Z$ \cite{Hi10}.
The meridianal automorphism induces an automorphism of 
$H_1(\pi';\mathbb{Z})=(\oplus*_{i\in{I}}H_1(G_i;\mathbb{Z}))\oplus{Z^s}$ 
which acts on $\oplus*_{i\in{I}}H_1(G_i;\mathbb{Z})$
by automorphisms and permutations of the summands.
Thus the non-free factors must admit homologically meridianal automorphisms,
by Lemma 2.
In particular, the factors which are virtually free but not free must be finite,
and thus must be the groups of $\mathbb{S}^3$-manifolds, 
by Theorem 15.12 of \cite{Hi}.
Therefore $M'$ is homotopy equivalent to a closed orientable 3-manifold.
The covering transformation $t$ corresponding to the meridian is an 
orientation-preserving self-homotopy equivalence, inducing an
outer automorphism class $[\theta]$.
Thus $(3)\Rightarrow(2).$

If (2) holds then we may realize $[\theta]$ 
by a self-homotopy equivalence of $N$ \cite{Sw}.
If $\pi'$ has no finite cyclic factors this is homotopic 
to a self-homeomorphism of $N$ \cite{HL,La}.
However if $\pi'$ has finite cyclic factors then 
we may have to modify $N$.
If we choose the lens space summands carefully, 
as in the construction following Lemma 2, 
then we may again realize $[\theta]$ by a self-homeomorphism.
Thus $(2)\Rightarrow(1).$
\end{proof}

If $\pi=\pi\widetilde{K}$ and $\pi'$ is finitely generated
then $\widetilde{K}$ need not be fibred.
If $K$ is fibred and $\pi'\cong\pi_1(N)$ has finite cyclic factors
the fibre need not be $N$.
However if $N$ is aspherical no such examples are known,
and under very plausible $K$- and $L$-theoretic hypotheses we may
obtain a stronger result.

\begin{theorem}
Let $M$ be a closed $4$-manifold with $\chi(M)=0$ and such
that $\pi=\pi_1(M)\cong\nu\rtimes{Z}$, 
where $\nu=\pi_1(N)$ for some aspherical closed $3$-manifold $N$.
If the ring $\mathbb{Z}[\nu]$ is coherent and 
the assembly map from $H_*(\nu;\mathbb{L}_0)$ to $L_*(\nu)$ 
is an isomorphism
then $M$ is $s$-cobordant to
the mapping torus of a self-homeomorphism of $N$.
\end{theorem} 

\begin{proof}
The manifold $M$ is also aspherical, 
and so the infinite cyclic covering space corresponding
to $\pi_1(N)$ is homotopy equivalent to $N$.
The generator of the covering group corresponds to
a self-homotopy equivalence of $N$. 
This is homotopic to a self-homeomorphism, 
and so $M$ is homotopy equivalent to a mapping torus $N\rtimes{S^1}$.

It follows from the Geometrization Conjecture 
and the work of Farrell and Jones that
$Wh(\nu)=\widetilde{K}_0(\mathbb{Z}[\nu])=0$.
If moreover the ring $\mathbb{Z}[\nu]$ is coherent then $Wh(\pi)=0$,
by the work of Waldhausen.
(It is well-known that 3-manifold {\it groups} are coherent.)

A comparison of Mayer-Vietoris sequences for the extension 
$\pi=\nu\rtimes{Z}$ shows that the assembly map from 
$H_*(\pi;\mathbb{L}_0)$ to $L_*(\pi)$ is also an isomorphism.
Since $M$ is aspherical it follows that $L_5(\pi)$ acts trivially 
on the $s$-cobordism structure set $S^s_{TOP}(M)$ (as defined in \S2
of Chapter 6 of \cite{Hi}), and so $M$ is $s$-cobordant to $N\rtimes{S^1}$.
\end{proof}

\begin{cor}
Let $K$ be a $2$-knot with group $\pi$ such that $\pi'\cong\pi_1(N)$,
where $N$ is an aspherical closed $3$-manifold.
If $\mathbb{Z}[\pi']$ is coherent and the assembly map is an isomorphism 
then $K$ is $s$-concordant to a fibred $2$-knot.
\qed
\end{cor} 

Does this corollary extend to torsion-free 2-knot groups 
$\pi$ with $\pi'$ finitely generated?
It does not hold when $\pi'\cong{Z/3Z}$. 
(See \S17.5 of \cite{Hi}.)

\section{high dimensional fibred knots}

Although Kervaire's characterization of high-dimensional knot groups
was one of the first results of high-dimensional knot theory,
there is apparently no corresponding characterization of the groups 
of fibred $n$-knots, and it is not clear whether the class of
such groups should be independent of $n$ for $n$ large.

Let $E_n$ be the set of $n$-knot groups $\pi$ with $\pi'$ finitely presentable,
and let $F_n$ be the set of groups of fibred $n$-knots.
It is easy to see that $E_1\subset{E_2}\subset{E_3}=E_n$
and $F_n\subseteq{E_n}$, for all $n\geq1$.
Moreover $F_n\subseteq{F_{n+1}}$, for all $n\geq1$,
since spins and superspins of fibred knots are fibred and 
these constructions preserve the knot group.
In low dimensions some of these inclusions are proper:
$F_1=E_1\subsetneq{F_2}\subsetneq{F_3}$ and $F_3\not\subseteq{E_2}$
(so $E_1\not=E_2\not=E_3$), 
since twist spins of prime knots are fibred but do not have free
commutator subgroup, and
there are fibred 3-knots with closed fibre the 4-torus \cite{CS},
whereas solvable 2-knot groups have Hirsch length 1, 2 or 4.

Theorem 2 suggests that $E_2$ and $F_2$ should agree.
We have the following weaker result.
If $K$ is a 2-knot such that $\pi=\pi{K}\in{E_2}$ then 
$M(K)'$ is a $PD_3$-complex, and so is finitely dominated \cite{HK}.
Hence $X(K)'$ is finitely dominated, 
and therefore so also are $M(\sigma^p{K})'$ and $M(S^p\otimes{K})'$,
where $\sigma^pK$ and $S^p\otimes{K}$ are the iterated spin 
and the $p$-superspin of $K$, respectively.
If $Wh(\pi)=0$ and $p\geq2$ these are fibred $(p+2)$-knots,
by the Farrell fibration theorem,
and so $\pi$ is the group of a fibred $n$-knot, for all $n\geq4$.

If $F_3=E_3$ then $F_n=E_3$ for all $n\geq3$. 
If not, there are a variety of weaker questions.
The fact that there are homology 5-sphere groups 
which are not homology 4-sphere groups
suggests to me that perhaps $F_3\not=F_4$.
On the other hand, I expect that $F_4=E_3$ and so $F_n=E_n=E_3$ for all
$n\geq4$.

There are high-dimensional knot groups $\pi$ with $\pi'$
finitely generated but not finitely presentable \cite{Si}.
It is not known whether there are any such 2-knot groups.

\section{cohomological dimension}

An $n$-knot group has cohomological dimension 1 
if and only if it is infinite cyclic.
If $\pi$ is a knot group with $\pi'$ finitely presentable
and $c.d.\pi=2$ then $\pi'$ is free, by Lemma 1.
In particular, $\pi'\not=\pi''$.
Spinning repeatedly (or superspinning) a nontrivial classical fibred 
knot gives fibred $n$-knots $K$ with such groups, for all $n$.
We shall give examples to show that for every $d>2$ and $n\geq4$ 
there is a fibred $n$-knot with group $\pi$ such that $c.d.\pi=d$ 
and $\pi'=\pi''$.
As observed earlier, if $\pi$ is the group of a fibred 2-knot 
then $v.c.d.\pi=1$, 2 or 4.
The corresponding result for fibred 3-knots remains unknown.

Let $H$ be the Higman superperfect group,
with presentation
\[ \langle{a,b,c,d}\mid bab^{-1}=a^2,~cbc^{-1}=b^2,~ dcd^{-1}=c^2,~
ada^{-1}=d^2\rangle\]
and $B$ be the group of the Brieskorn homology 3-sphere $\Sigma(2,3,7)$,
with presentation
\[
\langle{x,y,z}\mid x^2=y^3=z^7=xyz\rangle.\]
Then $H^k\times{B}^l$ is a finitely presentable, 
superperfect group of cohomological dimension $2k+3l$,
and so is the fundamental group of an homology $m$-sphere,
for all $k, l\geq0$ and all $m\geq5$.
The groups $H$ and $B$ have deficiency 0 and so are also
the groups of homology 4-spheres.
If $t$ is a generator for the $Z$ factor then the normal closure of the image
of $(a,\dots,a,x,\dots,x,t)$ in $H^k\times{B}^l\times{Z}$
is the whole group, and so such products have weight 1.
If $N$ is an homology $m$-sphere with group $H^k\times{B}^l$ 
then the cocore of surgery on a loop in $N\times{S^1}$ 
representing a normal generator gives a fibred $(m-1)$-knot 
with group $H^k\times{B}^l\times{Z}$.
In particular, there is a fibred $n$-knot $K$ with $c.d.\pi{K}=d$ 
and $\pi'=\pi''$ for every $d\not=2$ and $n\geq4$.
This construction shows also that $B\times{Z}$ 
is the group of a fibred 2-knot; in fact ${B}\times{Z}\cong\pi\tau_73_1$.

These groups are also 3-knot groups, 
but it is not clear whether they are all realized by 
{\it fibred\/} 3-knots.
Is every cohomological dimension other than 2 is realized
by some fibred 3-knot with perfect commutator subgroup?
Since the above construction gives fibred 3-knots with 
groups $H\times{Z}$ and ${B}\times{Z}$ it would suffice to show 
that for each $d>4$ there is a finitely presentable perfect group $G$ 
with deficiency 0, $c.d.G=d-1$ and
with an element $g\in{G}$ such that $\{[g,h]\mid{h\in{G}}\}$ 
has normal closure the whole group.
(There are ``Cappell-Shaneson" fibred 3-knots with 
$\pi\cong{Z^4\rtimes{Z}}$.
These have $c.d.\pi=5$, but $\pi''=1$!)

The groups of fibred 2-knots have solvable word problem,
since they are extensions of $Z$ by 3-manifold groups.
There is a 3-knot $K$ whose group is universal,
in the sense that every finitely presentable group 
is a subgroup of $\pi{K}$,
by Corollary 3.4 of \cite{GGS}.
In particular, $\pi{K}$ has unsolvable word problem.
There is a finitely presentable acyclic group $U$ which is universal.
Since $U$ is the fundamental group of an homology 5-sphere
there is a fibred 4-knot in an homology 6-sphere whose
knot group has commutator subgroup $U$.
Is there a fibred 3-knot whose knot group is universal?

If we drop the condition that the commutator subgroup 
of a 2-knot group be finitely generated we get new examples.
The group $\pi$ with presentation
\[\langle{a,b,t}\mid{tat^{-1}=a^2,~aba^{-1}=b^{-1},~b^3=1}\rangle\]
is the group of a satellite \cite{Ka83} of Fox's Example 10 around $\tau_23_1$
The image of $a$ in any finite quotient of $\pi$ must have odd order,
and so the image of $b$ must be trivial.
Therefore $\pi$ is not virtually torsion-free, and $v.c.d.\pi=\infty$.
Are there any 2-knot groups $G$ for which $v.c.d.G$ is finite,
but not 1, 2 or 4?

\section{finitely generated normal subgroups}

In this section we shall consider algebraic criteria for
other possible fibrations for 2-knot manifolds.

Let $K$ be a 2-knot whose group $\pi=\pi{K}$ has 
an infinite finitely generated infinite normal subgroup $H$
of infinite index.
Then $\beta_1^{(2)}(\pi)=0$ and $\pi$ has one end.
The following results are immediate consequences of
Theorems 4 and 6 of \cite{Hi08}.
\begin{enumerate}
\item 
If $\pi/H$ has two ends then $H$ has finite index in $\pi'$,
and so $\pi'$ is also finitely generated.
Hence $\pi'$ and $H$ are the fundamental groups of $PD_3$-spaces,
and are $FP_2$ \cite{HK}.

\item If $\pi/H$ has one end and $H$ is $FP_2$ then $M(K)$ is aspherical
and either $H\cong{Z}$ and $H^3(\pi/H;\mathbb{Z}[\pi/H])\cong{Z}$
or $H$ is a $PD_2$-group and 
$\pi/H$ is virtually a $PD_2$-group.

\item
If $\pi/H$ has infinitely many ends and $H$ is $FP_3$ then
$H^2(\pi;\mathbb{Z}[\pi])$ $\not=0$ and $M(K)$ is not aspherical.
In particular, $H$ has more than one end.
\end{enumerate}
If $M(K)$ is the total space of an orbifold fibration with fibre $F$ 
then $H=\pi_1(F)$ is $FP_3$.
Nevertheless, can the finiteness hypotheses on $H$ in (2) and (3) be relaxed?
If we consider instead knots $K$ such that $\pi$ 
has an ascendant $PD_2$ subgroup $H$ then a transfinite induction 
using the Lyndon-Hochschild-Serre spectral sequence shows that 
$H^s(\pi;\mathbb{Z}[\pi])=0$ for $s\leq2$, so $M$ is aspherical,
and then the only new possibity is that 
$\pi$ be virtually $(H\rtimes{Z})\rtimes{Z}$.
See Theorems 5 and 6 of \cite{Hi08}.

Most of these possibilities do occur.
If $K=\tau_63_1$ (the 6-twist spun trefoil)
or $\tau_2k(e,\eta)$ (a 2-twist spun Montesinos knot) 
then $\zeta\pi'\cong{Z}$, $\zeta\pi\cong{Z}^2$ and $\pi'$
are finitely presentable normal subgroups,
and in each case the quotient has one or two ends.
If $K=\sigma3_1$ (the spun trefoil) then $\pi'$ is free of rank 2,
while $\pi/\zeta\pi$ is virtually free of rank 2.
If $K=\tau_44_1$ then $M(K)$ is the mapping torus 
of a self-homeomorphism of a non-Haken hyperbolic 3-manifold 
which has a 10-fold cover which fibers over the circle, 
with the fiber having genus 2\cite{Re}. 
Thus $\pi$ is virtually $(H\rtimes{Z})\rtimes{Z}$ with $H=\pi_1(T\#T)$.

If all hyperbolic 3-manifolds are virtually fibred then
this structure is generic for $r$-twist spins of simple non-torus knots.
Is there such a knot for which some twist spin is a mapping torus? 
(The Alexander polynomial $\Delta_k(t)$ must be divisible by 
the cyclotomic polynomial $\phi_r(t)$, 
so $r$ must be composite and $k$ must have at least 8 crossings.)

In more detail:
the considerations of \S2 above apply to 2-knots in case (1).

It is not known whether a finitely presentable group $G$ such that
$H^3(G;\mathbb{Z}[G])\cong{Z}$ must be virtually a $PD_3$-group.
However we have the following complement to case (2).

\begin{theorem}
Let $\pi$ be a $2$-knot group such that $\pi'$ is finitely generated
and with an infinite cyclic normal subgroup $H$
such that $\pi/H$ has one end.
Then $\pi/H$ is virtually a $PD_3$-group.
If moreover $H<\pi'$ then $M(K)$ is homotopy equivalent to the
mapping torus of a self-homeomorphism of an aspherical Seifert fibred
$3$-manifold.
\end{theorem}

\begin{proof}
If $K$ is a 2-knot with group $\pi$ then $M(K)$ is aspherical,
and so $\pi$ is a $PD_4$-group.
Since $\pi'$ is finitely generated it is a $PD_3$-group \cite{HK}.

If $H\cap\pi'=1$ then $(\pi/H)'\cong\pi'$ and
$[\pi/H:(\pi/H)']=[\pi:H\pi']$ is finite,
so $\pi/H$ is virtually a $PD_3$-group.

If $H<\pi'$ then $H\leq\zeta\pi'$, 
and so $\pi'$ is the group of an aspherical Seifert fibred 3-manifold 
\cite{Bow}.
In particular, $\pi'/H$ is virtually a $PD_2$-group,
and so $\pi/H$ is again virtually a $PD_3$-group.
\end{proof}

The quotient $\pi/H$ need not be orientable,
even if it is a 3-manifold group. See \S15.3 of \cite{Hi}.
Does the lemma hold whenever $H\cong{Z}$ and $\pi/H$ has one end?
If $H\cong{Z}$ and $\pi/H$ is virtually a $PD_3$-group then $H$
has finite index in a maximal infinite cyclic normal subgroup.

The other possibility in case (2) is realized by orbifold bundle spaces.

\begin{theorem}
A group $\pi$ is the group of a $2$-knot $K$ such that $M(K)$
is the total space of an orbifold fibration with aspherical, 
$2$-dimensional base $B$ and fibre $F$ if and only if
it is a torsion-free extension of $\beta=\pi_1^{orb}(B)$ 
by $\phi=\pi_1(F)$, $\phi$ is a $PD_2^+$-group,
$\beta$ acts on $H_2(\phi;\mathbb{Z})$ through $w_1(\beta)$,
$\pi$ has weight $1$ and $\chi(F)\chi^{orb}(B)=0$.
If these conditions hold then 
$\pi/\pi'\phi\cong\beta/\beta'$ is finite cyclic, 
so $\phi\not\leq\pi'$, and $\zeta\pi\leq\phi$.
\end{theorem}

\begin{proof}
The necessity of the algebraic conditions is clear.
If they hold then $\pi$ is the fundamental group of an
orientable orbifold bundle space $M$, by Theorem 7.3 of \cite{Hi}.
Since $M$ is orientable and $\chi(M)=0$ the cocore of
surgery on a normal generator for $\pi$ is a 2-knot
$K$ with $M(K)=M$.

The final assertions hold since
no 2-orbifold group has abelianization $Z$,
and the centre of a 2-orbifold group with cyclic abelianization is trivial.
\end{proof}

\begin{cor}
If $\chi(\beta)=0$ then $B=S^2(2,3,6)$ 
(the $2$-sphere with three cone points),
$\mathbb{D}^2(\overline{3},\overline{3},\overline{3})$
(the $2$-disc with reflector boundary and three corner points) 
or $\mathbb{D}^2(3,\overline{3})$
(the $2$-disc with reflector boundary, one corner point and one cone point). 
\end{cor}

\begin{proof}
These are the flat 2-orbifolds for which $\beta/\beta'$ is cyclic.
\end{proof}

\begin{cor}
If $\chi(\phi)=0$ then $\phi$
is the unique maximal normal $PD_2^+$-subgroup of $\pi$.
\end{cor}

\begin{proof}
If $\tilde\phi$ is another maximal normal $PD_2^+$-subgroup of $\pi$
then $\phi\tilde\phi$ is a normal subgroup which properly contains
$\phi$ and so $[\phi\tilde\phi:\phi]=\infty$.

If $\chi(\tilde\phi)=0$ also then $\phi(\tilde\phi\cap{C_\pi(\phi))}$ 
is abelian, of rank at least 3.
But then either $\pi'\cong{Z^3}$ or $\pi$ is virtually $Z^4$,
and no such 2-knot group has an abelian normal subgroup of rank 2.
(See \S16.4 of \cite{Hi}.)

If $\chi(\tilde\phi)\not=0$ then $\phi\cap\tilde\phi=1$,
and the image of $\phi$ in $\tilde\beta=\pi/\tilde\phi$
has finite index and centralizer of index $\leq2$,
since $[\pi:C_\pi(\phi)]\leq2$, by Theorem 16.2 of \cite{Hi}.
But then the holonomy of $\tilde\beta$ has order at most 2,
contrary to $\beta/\beta'$ being cyclic.
\end{proof}

If $K$ is a $2$-knot such that $\pi{K}$ has an abelian normal subgroup
of rank $2$ then $M(K)$ is $s$-cobordant to a 
${\widetilde{\mathbb{SL}}\times\mathbb{E}^1}$-manifold or is
homeomorphic to one of the $\mathbb{N}il^3\times\mathbb{E}^1$-manifolds
$M(\tau_63_1)$ or $M(\tau_2k(e,\eta))$, 
for some even $e$ and $\eta=\pm1$.
The manifolds $M(\tau_63_1)$ and $M(\tau_2k(e,\eta))$ are 
Seifert fibred over $S^2(2,3,6)$ and 
$\mathbb{D}^2(\overline{3},\overline{3},\overline{3})$,
respectively.
Consideration of the possible knot groups shows that no 2-knot manifold is
Seifert fibred over $\mathbb{D}^2(3,\overline{3})$.
(See Chapter 16 of \cite{Hi}.)

At present, no such examples with $\chi(\phi)\not=0$ have been found.
What little we know is summarized in the following theorem.

\begin{theorem}
If $K$ is a $2$-knot such that $\pi=\pi{K}$ is an extension of a
flat $2$-orbifold group $\beta$ by a normal subgroup $\phi$ which is
a $PD_2^+$-group then $\chi(\phi)\equiv0$ {\it mod\/} $(6)$
and $\pi/\phi\cap\pi'$ is a $3$-dimensional crystallographic group.
The group $\pi$ has no non-trivial abelian normal subgroup.
\end{theorem}

\begin{proof}
The torsion subgroup of $\beta$ is isomorphic to
$Z/6Z$ or $S_3$, by the Corollary to Theorem 6.
The preimage in $\pi$ of this torsion subgroup 
is a torsion-free extension of $\phi$,
and so $\chi(\phi)\equiv0$ {\it mod\/} (12),
if $\beta$ is orientable,
and $\chi(\phi)\equiv0$ {\it mod\/} (6) otherwise.

Since $\beta/\beta'$ is finite $\phi/\phi\cap\pi'\cong{Z}$,
and so $\phi\cap\pi'$ is free of countable rank.
The preimage in $\pi/\phi\cap\pi'$ of the translation subgroup of $\beta$
is isomorphic to $Z^3$, 
and $\pi/\phi\cap\pi'$ has no non-trivial finite normal subgroup.
Thus $\pi/\phi\cap\pi'$ is a 3-dimensional crystallographic group.

If $A$ is an abelian normal subgroup of $\pi$ then $A\cap\phi=1$,
and so $A$ maps injectively to an abelian normal subgroup of $\beta$. 
Therefore if $A$ is non-trivial then $A\cong{Z^2}$.
But this is impossible, by the second corollary of Theorem 7.
\end{proof}

\begin{cor}
The knot manifold $M(K)$ is not Seifert fibred, and $K$ is not a twist spin.
\qed
\end{cor}

In particular, if $B=S^2(2,3,6)$ then ${\pi/\phi\cap\pi'\cong{G_5}}$, 
the orientable flat 3-manifold group with holonomy $Z/6Z$.
Are there any examples with $\pi'$ finitely generated (but not solvable)?
If so $\pi'$ would be a $PD_3$-group with free commutator subgroup.
Must such a group be a semidirect product $H\rtimes_\theta{Z}$
with $H$ a $PD_2$-group?
If $H$ is hyperbolic, $\theta$ must have infinite order in $Out(H)$.

Are there any examples with base $B=\mathbb{D}^2(3,\overline{3})$
and hyperbolic fibre?

In case (3)
are there examples with both $H$ and $\pi/H$ having infinitely many ends?
In particular, are then any such with $H\cong{F(r)}$ for some $r>1$?


\newpage
\begin{thebibliography}{99}

\bibitem{Bi} Bieri, R. {\it Homological Dimensions of Groups},

Queen Mary College Mathematics Notes, London (1976).

\bibitem{Bow} Bowditch, B.H. Planar groups and the Seifert conjecture,

J. Reine Angew. Math. 576 (2004), 11--62.

\bibitem {CS} Cappell, S.E. and Shaneson, J.L. There exist 
inequivalent knots with the 

same complement,
Ann. Math. 103 (1976), 349--353.

\bibitem{Cr} Crisp, J.S. The decomposition of Poincar\'e duality complexes,

Comment. Math. Helv. 75 (2000), 232--246.

\bibitem{FJ} Farrell, F.T. and Jones, L.P. Isomorphism Conjectures in algebraic
$K$-theory,

J.Amer. Math. Soc. 9 (1993), 249--295.

\bibitem{DF} Fried, D. and Lee, R. Realizing group automorphisms,

in {\it Group actions on manifolds (Boulder, Colo., 1983)},
CONM 36, 

Amer. Math. Soc., Providence, RI, (1985), 427--432.

\bibitem{Gi} Gilbert, N.D. Presentations of the automorphism group of a free
product,

Proc. London Math. Soc. 54 (1987), 115--140.

\bibitem{GGS} Gonz\'alez-Acu\~na, F., Gordon, C. McA. and Simon, J.
Unsolvable problems about higher-dimensional knots and related groups,
arXiv.math.GR 0908.4009.


\bibitem{HL} Hendriks, H. and Laudenbach, F. Scindement d'une \'equivalence 
d'homotopie en dimension 3, 
Ann. Sci. Ecole Norm. Sup. 7 (1974), 203--217.

\bibitem{Hi} Hillman, J.A. {\it Four-Manifolds, Geometries and Knots},

Geometry and Topology Monographs, vol. 5,

Geometry and Topology Publications (2002).

\bibitem{Hi08}  Hillman, J.A. 
Finitely dominated covering spaces of $3$- and $4$-manifolds,

J. Austral. Math. Soc. 84 (2008), 99--108.

\bibitem{Hi10}  Hillman, J.A. Indecomposable $PD_3$-complexes,

arXiv.math.GT 0808.1775.

\bibitem{HK} Hillman, J.A. and Kochloukova, D.H.  Finiteness conditions 
and $PD_r$-group covers of $PD_n$-complexes,
Math. Z. 256 (2007), 45--56.

\bibitem{Ka83} Kanenobu, T. Groups of higher dimensional satellite knots,

J.Pure Appl. Alg. 28 (1983), 179--188.

\bibitem{Ko} Kochloukova, D. H.  On a conjecture of E.Rapaport Strasser about
knot-like groups and its pro-$p$ version,
J. Pure App. Algebra 204 (2006), 536--554.

\bibitem{La} Laudenbach, F. {\it Topologie de la Dimension Trois: Homotopie et Isotopie},

Ast\'erisque 12 (1974).

\bibitem{Re} Reid, A. W. A non-Haken hyperbolic $3$-manifold covered by a surface bundle,

Pacific J. Math. 167 (1995), no. 1, 163--182.

\bibitem{Si} Silver, D. Examples of 3-knots with no minimal Seifert manifolds,

Math. Proc. Cambridge Philos. Soc. 110 (1991), 417--420.

\bibitem{Sw} Swarup, G.A. On a theorem of C.B.Thomas,

J. London Math Soc. 8 (1974), 13--21.

\end{thebibliography}
\end{document}